\documentclass[11pt,twoside]{amsart}
\usepackage[latin1]{inputenc}
\usepackage[OT1]{fontenc}
\usepackage{amsmath}
\usepackage{amsthm}
\usepackage{amssymb}
\usepackage[all]{xy}
\usepackage{bbm}
\usepackage{amscd}
\usepackage{a4wide}
\usepackage{enumitem}

\DeclareMathOperator{\pr}{pr}
\DeclareMathOperator{\id}{id}

\DeclareMathOperator{\Hom}{Hom}

\DeclareMathOperator{\sHom}{\mathcal{H}\textit{om}}
\DeclareMathOperator{\Ext}{Ext}

\DeclareMathOperator{\Coh}{Coh}

\DeclareMathOperator{\Ho}{H}

\DeclareMathOperator{\Hilb}{Hilb}
\DeclareMathOperator{\GHilb}{G-Hilb}

\DeclareMathOperator{\rank}{\mathsf{rank}}

\DeclareMathOperator{\cone}{cone}

\DeclareMathOperator{\Ind}{\mathsf{Ind}}

\DeclareMathOperator{\Dperf}{D^{\mathrm{perf}}}

\newcommand{\D}{{\rm D}}

\DeclareMathOperator{\Aut}{\mathrm{Aut}}

\DeclareMathOperator{\Pic}{\mathrm {Pic}}

\newcommand{\cI}{{\mathcal I}}

\newcommand{\IA}{\mathbb{A}}
\newcommand{\IC}{\mathbb{C}}

\newcommand{\IN}{\mathbb{N}}
\newcommand{\IP}{\mathbb{P}}

\newcommand{\IZ}{\mathbb{Z}}

\DeclareMathOperator{\Res}{\mathsf{Res}}

\DeclareMathOperator{\FM}{\mathsf{FM}}
\DeclareMathOperator{\TT}{\mathsf{T}\!}
\DeclareMathOperator{\MM}{\mathsf{M}}

\DeclareMathOperator{\CC}{\mathsf{C}}

\DeclareMathOperator{\WW}{\mathsf{W}}
\DeclareMathOperator{\stab}{\mathsf{stab}}

\newcommand{\sym}{\mathfrak S}

\newcommand{\cF}{\mathcal F}

\newcommand{\cE}{\mathcal E}

\newcommand{\cZ}{\mathcal Z}

\newcommand{\cD}{\mathcal D}

\newcommand{\cK}{\mathcal K}

\newcommand{\alt}{\mathfrak a}

\newcommand{\reg}{\mathcal O}

\newcommand{\eps}{\varepsilon}

\renewcommand{\theta}{\vartheta}
\renewcommand{\rho}{\varrho}
\renewcommand{\phi}{\varphi}
\renewcommand{\_}{\underline{\,\,\,\,}}

\usepackage[usenames,dvipsnames]{xcolor}
\usepackage{hyperref}
\hypersetup{
    colorlinks,
    linkcolor={red!50!black},
    citecolor={blue!50!black},
    urlcolor={blue!80!black}
}
\usepackage{aliascnt} 

\newtheorem{theorem}{Theorem}[section]

  \newaliascnt{proposition}{theorem}
  \newtheorem{prop}[proposition]{Proposition}
  \aliascntresetthe{proposition}

  \newaliascnt{lemma}{theorem}
  \newtheorem{lemma}[lemma]{Lemma}
  \aliascntresetthe{lemma}

  \newaliascnt{corollary}{theorem}
  \newtheorem{cor}[corollary]{Corollary}
  \aliascntresetthe{corollary}

\theoremstyle{definition}

  \newaliascnt{definition}{theorem}
  \newtheorem{definition}[definition]{Definition}
  \aliascntresetthe{definition}

  \newaliascnt{remark}{theorem}
  \newtheorem{remark}[remark]{Remark}
  \aliascntresetthe{remark}

  \newaliascnt{condition}{theorem}
  
  \aliascntresetthe{condition}

  \newaliascnt{question}{theorem}
  
  \aliascntresetthe{question}

  \newaliascnt{example}{theorem}
  
  \aliascntresetthe{example}

\begin{document}

\title[McKay correspondence for Hilbert schemes and tautological bundles]{Remarks on the derived McKay correspondence for Hilbert schemes of points and tautological bundles}
\author[A.\ Krug]{Andreas Krug}

\begin{abstract}
We study the images of tautological bundles on Hilbert schemes of points on surfaces and their wedge powers under the derived McKay correspondence. The main observation of the paper is that using a derived equivalence differing slightly from the standard one considerably simplifies both the results and their proofs. As an application, we obtain shorter proofs for known results as well as new formulae for homological invariants of tautological sheaves. In particular, we compute the extension groups between wedge powers of tautological bundles associated to line bundles on the surface.  
\end{abstract}

\maketitle

\section{Introduction}

Let $G$ be a finite group which acts on a smooth variety $M$. The McKay correspondence is a principle describing the relationship between the geometry of certain resolutions of the singularities of the quotient $M/G$ and the representation theory of $G$. Probably the most important example of the McKay correspondence in higher dimensions is the case where $M=X^n$ is a power of a smooth surface with the symmetric group $G=\mathfrak S_n$ permuting the factors. In this case, a crepant resolution of the quotient singularities is given by the Hilbert scheme $X^{[n]}$ of points on $X$ which is a fine moduli space of zero-dimensional subschemes of $X$. The McKay correspondence can then be expressed as an equivalence of derived categories $\D(X^{[n]})\cong \D_{\sym_n}(X^n)$ of ($\sym_n$-equivariant) coherent sheaves; see \cite{BKR, Hai}.

Besides being a very interesting theoretical result, the derived McKay correspondence can be used as a computational tool for the study of vector bundles, or, more generally, sheaves and complexes thereof, on the Hilbert schemes of points on surfaces. Concretely, given a vector bundle on $X^{[n]}$, the derived McKay correspondence $\D(X^{[n]})\cong \D_{\sym_n}(X^n)$ can be used to translate this bundle into a complex in $\D_{\sym_n}(X^n)$. Then, the homological invariants of the vector bundle agree with those of the associated equivariant complex but the computations are often easier for the latter.

A very interesting class of vector bundles on $X^{[n]}$ is given by the tautological bundles $E^{[n]}$. They are associated to vector bundles $E$ on the surface $X$ by means of the universal family of the Hilbert scheme; see \autoref{def:taut} for details. These bundles were intensively studied for various reasons. First of all, it seems natural to consider tautological bundles if one is interested in the geometry of Hilbert schemes of points on surfaces.
Furthermore, they have applications in the description of the cup product on the cohomology of the Hilbert scheme \cite{LehnChern, LehnSorger1, LehnSorger2}, enumerative geometry \cite{KoolShendeThomas, RennemoTaut}, and the strange duality conjecture for line bundles on moduli spaces of sheaves \cite{Dandual, MarianOpreathetatour}. Recently, they have also been considered as a source of examples of stable bundles in higher dimension; see \cite{Schlickeweitaut,Wandel1,Wandel2,Stapletontaut}.

In \cite{Sca1}, Scala began to use the derived McKay correspondence to study tautological bundles and tensor powers thereof. In particular, he explicitly computed equivariant complexes in $\D_{\sym_n}(X^n)$ corresponding to the tautological sheaves. This has been further exploited in \cite{Sca2, Scasymmetric,  KruExt, KrugTensortaut, Mea, MMdefK3}. In the present paper, we consider an equivalence $\D(X^{[n]})\xrightarrow\cong \D_{\sym_n}(X^n)$ which differs slightly from the one used in \cite{Sca1} and the subsequent papers. The main observation is that this considerably simplifies the description of the images of tautological bundles and their wedge powers under the equivalence as well as the proofs of these descriptions. As an application, we get new formulae for extension groups and Euler characteristics of bundles on the Hilbert scheme as well as simplified proofs of known formulae. 

Let us describe the results of this paper in more detail. The main point in establishing the derived McKay correspondence $\D(X^{[n]})\cong \D_{\sym_n}(X^n)$ is the identification, due to \cite{Hai}, of the Hilbert scheme $X^{[n]}$ of points on $X$ with  the fine moduli space of $\sym_n$-clusters on $X^n$. These $\sym_n$-clusters are, roughly speaking, 
scheme-theoretic generalisations of free $\sym_n$-orbits; see \autoref{subsect:derivedMcKay} for some more details. In particular, there is a universal family of $\sym_n$-clusters $\cZ\subset X^{[n]}\times X^n$ together with the projections
\[
X^{[n]}\xleftarrow q \cZ \xrightarrow p X^n\,. 
\]
Then the 'usual' derived McKay correspondence, as considered in \cite{BKR, Sca1, Sca2, Scasymmetric,  KruExt, KrugTensortaut, Mea, MMdefK3}, is the equivalence of derived categories
\[
 \Phi:=Rp_*\circ q^*\colon \D(X^{[n]})\xrightarrow \cong \D_{\sym_n}(X^n)\,.
\]
In \cite{Sca1}, the image of a tautological bundle under the derived McKay correspondence is described by the formula \begin{align}\label{formula:Scala}\Phi(F^{[n]})\cong \CC^\bullet_F\,.\end{align} Here, $\CC^\bullet_F$ is a complex of $\sym_n$-equivariant coherent sheaves on $X$ concentrated in degree zero with
\[
\CC^0_F=\bigoplus_{i=1}^n \pr_i^*F 
\]
where $\pr_i\colon X^n\to X$ is the projection to the $i$-th factor; see \autoref{subsection:Scala} for details on the higher degree terms of $\CC^\bullet_F$. 
The formula \eqref{formula:Scala} has been used in \cite{Sca1, Sca2, Scasymmetric,  KruExt, KrugTensortaut, Mea, MMdefK3} in order to prove many interesting consequences. However, the proofs are often computationally involved, mainly due to the higher degree terms of the complex $\CC^\bullet_F$.

The main observation exploited in this paper is that it has benefits to consider the derived McKay correspondence in the reverse direction
\[
 \Psi:=(\_)^{\sym_n}\circ q_*\circ Lp^*\colon \D_{\sym_n}(X^n)\to \D(X^{[n]})
\]
instead. The functor $\Psi$ is again an equivalence, but not the inverse of $\Phi$; see \autoref{prop:Psiequi}.
The technical main result is that, if we replace $\Phi$ by $\Psi^{-1}$, the higher order terms of $\CC^\bullet_F$ vanish and we get a similarly simple description for the images of wedge powers of tautological bundles associated to line bundles on the surface.
\begin{theorem}[\autoref{thm:PsiCC}, \autoref{thm:wedgeline}]\label{thm:intro1}\noindent
\begin{enumerate}
 \item For every coherent sheaf $F\in \Coh X$, we have 
$\Psi(\CC^0_F)\cong F^{[n]}$.
\item For every line bundle $L\in\Pic X$ and $0\le k\le n$ , we have 
\[
 \Psi(\WW^k(L))\cong \wedge^kL^{[n]} \quad\text{where}\quad \WW^k(L)=\bigoplus_{\substack{I\subset\{1,\dots, n\}\\ |I|=k}} \pr_I^*(L^{\boxtimes k})\,. 
\]
Here, $\pr_I\colon X^n\to X^k$ is the projection to the $I$-factors and $\WW^k(L)$ carries a $\sym_n$-linearisation by permutation of the direct summands together with appropriate signs; see \autoref{def:CW} for details.  
\end{enumerate}
\end{theorem}
Objects of the form $\WW^k(L)$ play an important role in the construction of exceptional sequences \cite{KSos} and a categorical Heisenberg action \cite{CL, KrugHeisenberg} on the equivariant derived category 
$\D_{\sym_n}(X^n)$. Hence, \autoref{thm:intro1} can be seen as a step towards a geometric interpretation of these categorical constructions in terms of $\D(X^{[n]})$; see \autoref{rem:Heisenberginterpretation} for a few more details on this point of view. 

However, we will mainly use \autoref{thm:intro1} as a tool to compute homological invariants of tautological bundles and their wedge powers. We are able to give proofs of most of the known results on the cohomology and extension groups of these bundles which are much simpler than the original ones of \cite{Sca1, Sca2, KruExt}. Furthermore, we obtain new formulae such as
\begin{theorem}[\autoref{cor:Extformulae}]\label{thm:Extintro}
For $K,L\in \Pic X$ there are functorial isomorphisms
\begin{align*}
\begin{aligned}
&\Ext^*(\wedge^k K^{[n]}, \wedge^\ell L^{[n]})\\\cong& \bigoplus_{i=\max \{0,k+\ell-n\}}^{\min\{k,\ell\}}S^i\Ext^*(K,L)\otimes \wedge^{k-i}\Ho^*(K^\vee)\otimes \wedge^{\ell-i}\Ho^*(L)\otimes S^{n+i-k-\ell}\Ho^*(\reg_X)\,.
\end{aligned}
\end{align*}
\end{theorem}

In \autoref{prop:relativeFMtensor} and \autoref{cor:relativeFMcoh} we observe that it can be very useful for the computation of tensor products of bundles on the Hilbert scheme to have descriptions of their images under both, $\Phi$ and $\Psi^{-1}$. As an application, we prove the formula 
 \[
\sum_{n=0}^\infty\chi(F^{[n]}\otimes \Lambda_u L^{[n]})Q^n=\frac{(1+uQ)^{\chi(L)}}{(1-Q)^{\chi(\reg_X)}}\cdot \sum_{p=1}^\infty (-1)^{p+1}\chi\bigl(F\otimes (L^{p-1}u^{p-1}+L^pu^p)\bigr)Q^p 
 \]
for a generating function of the Euler characteristics where $F\in \Coh X$ and $L\in \Pic X$; see \autoref{thm:tensorEuler} and \autoref{rem:tensorEuler}. Here, for a vector bundle $E$ of rank $r$ and a formal parameter $t$, we use the notational convention $\Lambda_{t}E:=\sum_{i=0}^r(\wedge^i E)t^i$ as a sum in the Grothendieck group.

\autoref{thm:Extintro} is a generalisation and strengthening of the formula for the Euler bicharacteristics 
\begin{align}\label{Lcountingintro}
\sum_{n=0}^\infty \chi(\Lambda_{-v}L^{[n]}, \Lambda_{-u} L^{[n]})Q^n= \exp\left(\sum_{r=1}^\infty \chi(\Lambda_{-v^r} L, \Lambda_{-u^r} L) \frac{Q^r}r    \right) 
\end{align}
of \cite{WangZhouTaut}; see \autoref{Appendix} for details. In \textit{loc.\ cit.\ }formula \eqref{Lcountingintro} is conjectured to hold in greater generality. 
In \autoref{sec:further}, we give some restrictions to this conjecture showing that it does not hold if we replace the surface $X$ by a curve, neither if we replace the line bundle $L$ by a vector bundle of higher rank. 
In \autoref{prop:curvecase}, we also do some further computations concerning tautological bundles on Hilbert schemes of points on curves. 

\medskip
\noindent
\textbf{Acknowledgements.}
The author thanks J\"org Sch\"urmann for interesting discussions and S\"onke Rollenske for comments on the text.
\section{Preliminaries}
\subsection{General conventions}
All our varieties are connected and defined over the complex numbers $\IC$. 
For $M$ a variety, $\D(M):=\D^b(\Coh(M))$ denotes the bounded derived category of coherent sheaves.
We do not distinguish in the notation between a functor between abelian categories and its derived functor. For example, if $f\colon X\to Y$ is a morphism, we will write $f^*\colon \D(Y)\to \D(X)$ instead of $Lf^*$ for the derived pull-back.
\subsection{Equivariant sheaves and derived categories}

Let $G$ be a finite group acting on a variety $M$. We denote by $\Coh_G(M)$ the abelian category of equivariant coherent sheaves and by $\D_G(M):=\D^b(\Coh_G(M))$ its bounded derived category. In this section, we collect some facts about equivariant categories and functors that we need later. We refer to \cite[Sect.\ 4]{BKR} or \cite{Elagin} for further details.

Let $H\subset G$ be a subgroup. The forgetful (also called \textit{restriction}) functor $\Res_G^H\colon \Coh_G(M)\to \Coh_H(M)$ has a both-sided adjoint, namely the \textit{induction functor} $\Ind_H^G\colon \Coh_H(M)\to \Coh_G(M)$. Concretely, for $E\in \Coh(M)$, we have $\Ind_H^G(E)=\oplus_{g\in G/H} g^*E$ equipped with a $G$-linearisation which combines the $H$-linearisation of $E$ with appropriate permutations of the direct summands.  

In our case, the group $G$ will usually be the symmetric group $\sym_n$. We denote its non-trivial character by $\alt$ or $\alt_n$. We get an autoequivalence $\_\otimes \alt\colon \Coh_{\sym_n}(M)\to \Coh_{\sym_n}(M)$ given by changing the sign of the linearisations appropriately (of course, there is also an endofunctor $\_\otimes \rho\colon \Coh_G(M)\to \Coh_G(M)$ for an arbitrary representation $\rho$ of a finite group $G$). 

Let $G$ act on a second smooth variety $N$ and let $f\colon M\to N$ be a $G$-equivariant morphism. Then pull-backs and push-forwards of equivariant sheaves inherit canonical linearisations so that we get functors $f^*\colon \Coh_G(N)\to \Coh_G(M)$ and, if $f$ is proper, $f_*\colon \Coh_G(M)\to \Coh_G(N)$. Furthermore, there are equivariant tensor products and homomorphism sheaves.

Restriction, induction, and tensor products by representations commute with equivariant pull-backs and push-forwards which means that we have the following isomorphisms of functors
\begin{align}\label{Indcommutes}
\begin{aligned}
 \Res f^*\cong f^*\Res\,, \quad \Res f_*\cong f_*\Res\,,\quad \Ind f^*\cong f^*\Ind \,,\quad \Ind f_*\cong f_*\Ind\,,\\ f_*(\_\otimes \rho)\cong f_*(\_)\otimes \rho\quad,\quad f^*(\_\otimes \rho)\cong f^*(\_)\otimes \rho\,.
\end{aligned}
 \end{align}

All the functors discussed above induce functors on the level of the derived categories. We write these induced functors in the same way as the functors between the abelian categories, e.g.\ we write $f_*\colon \D_G(M)\to \D_G(N)$ instead of $Rf_*\colon \D_G(M)\to \D_G(N)$.

For two objects $E,F\in \D_G(M)$, we denote the graded Hom-space by 
\[\Hom_G^*(E,F):=\oplus_{i\in \IZ}\Hom_G^i(E,F)\quad \text{where}\quad \Hom^i_G(E,F):=\Hom_{\D_G(M)}(E,F[i])\,.\] We often suppress the restriction functor in the notation writing $E:=\Res E\in \D(M)$ for $E\in \D_G(M)$. The Hom-space $\Hom^i(E,F):=\Hom_{\D(M)}(\Res E, \Res F[i])$ has a canonical $G$-action induced by the $G$-linearisations of $E$ and $F$ and the invariants under this action are the Hom-spaces in the equivariant category:
\begin{align}\label{GinvaHom}
\Hom_G^*(E,F)\cong \Hom^*(E,F)^G\,. 
\end{align}

If $G$ acts trivially on $M$, a $G$-equivariant sheaf $E$ is simply a sheaf together with a group action. Hence, we can take the invariants of $E(U)$ for every open subset $U\subset X$, which gives a functor $(\_)^G\colon \Coh_G(M)\to \Coh(M)$. If $f\colon M\to N$ is a morphism between varieties on which $G$ acts trivially, we have
\begin{align}\label{invacommutes}
 (\_)^Gf_*\cong f_*(\_)^G\quad,\quad (\_)^Gf^*\cong f^*(\_)^G\,.  
\end{align}
Let $G$ act on a variety $S$ and let $\pi\colon S\to T$ be a $G$-invariant morphism of varieties. In other words, $\pi$ is $G$-equivariant when we consider $G$ acting trivially on $T$. Then we write 
\[
 \pi_*^G:=(\_)^G\circ \pi_*\colon \Coh_G(S)\to \Coh(T)\,.
\]
Furthermore, we simply write $\pi^*\colon \Coh(T)\to \Coh_G(S)$ for the functor which first equips every sheaf with the trivial $G$-action and then applies the equivariant pull-back.

\begin{lemma}\label{lem:quotientff}
Let $\pi \colon M\to M/G$ be the quotient of a smooth variety by a finite group. Then we have $\pi_*^G\pi^*\cong \id$ as endofunctors of the subcategory of perfect complexes $\Dperf(M/G)\subset\D(M/G)$. 
\end{lemma}
\begin{proof}
By definition of the quotient variety, we have $\pi_*^G\reg_M\cong \reg_{M/G}$. Now, the assertion follows by the (equivariant) projection formula. 
\end{proof}
\begin{remark}
We need to restrict the pull-back to the category of perfect complexes since, if the quotient $M/G$ is not smooth, $\pi^*$ does not preserve the bounded derived category.
\end{remark}

%
%

Let $G$ act trivially on $M$.
The following is a straight-forward generalisation of Frobenius reciprocity; compare \cite[Lem.\ 2.2]{Dan} or \cite[Sect.\ 3.5]{Kru4}.
\begin{lemma}\label{lem:globalFrob}
There is an isomorphism of functors $(\_)^G\Ind_U^G\cong (\_)^U\colon \D_U(M)\to \D(M)$.
\end{lemma}
\begin{remark}\label{rem:Danila}
One direct consequence of the lemma is the following. Let $E=\oplus_{i\in \cI}E_i\in \D(X)$ be a finite direct sum. Assume $E$ has a $G$-linearisation $\lambda$ and there is a $G$-action on $\cI$ such that $\lambda_g(E_i)=E_{g(i)}$. We say that $\lambda$ \textit{induces} the action on the index set $\cI$. Let $\{i_1,\dots,i_k\}$ be a set of representatives of the $G$-orbits of $\cI$ and set $G_j:=\stab_G(i_j)$ for $j=1,\dots, k$. Then 
\[
E^G\cong \bigoplus_{j=1}^k E_{i_j}^{G_j}\,. 
\]
\end{remark}

\subsection{Hilbert schemes of points and tautological sheaves}

Throughout the text, $X$ will be a smooth quasi-projective surface. For a non-negative integer $n\in \IN$, we denote by $X^{[n]}$ the Hilbert scheme of $n$ points on $X$. It is the fine moduli space of closed zero-dimensional subschemes of length $n$ of $X$. Hence, there is a universal family $\Xi\subset X^{[n]}\times X$ which is flat and finite of degree $n$ over $X$. The Hilbert scheme $X^{[n]}$ is smooth of dimension $2n$; see \cite[Thm.\ 2.4]{Fog}. 

The symmetric group $\sym_n$ acts on the cartesian power $X^n$ by permutation of the factors. We call the quotient $X^{(n)}:=X^n/\sym_n$ the \textit{$n$-th symmetric power} of $X$. We write the points of $X^{(n)}$ as formal sums of points of $X$. More concretely, $x_1+\dots+x_n\in X^{(n)}$ is the point lying under the $\sym_n$-orbit of $(x_1,\dots, x_n)\in X^n$. The Hilbert scheme is a resolution of the singularities of the symmetric power via the \textit{Hilbert--Chow morphism}
\[
 \mu\colon X^{[n]}\to X^{(n)}\quad, \quad [\xi]\mapsto \sum_{x\in \xi} \ell(\reg_{\xi,x})\cdot x\,.
\]
which sends a zero-dimensional subscheme to its weighted support.  
\begin{definition}\label{def:taut}
Let $\pr_X\colon \Xi\to X$ and $\pr_{X^{[n]}}\colon \Xi\to X^{[n]}$ be the projections from the universal family. 
We define the \textit{tautological functor} by 
\[
(\_)^{[n]}:=\pr_{X^{[n]}*}\pr_X^*\colon \D(X)\to \D(X^{[n]})\,. 
\]
Equivalently, $(\_)^{[n]}\cong \FM_{\reg_{\Xi}}$ can be written as the Fourier--Mukai transform along the structure sheaf of the universal family.
For $F\in \D(X)$, its image $F^{[n]}\in \D(X^{[n]})$ under the tautological functor is called the \textit{tautological object} associated to $F$.
\end{definition}
Let $E$ be a vector bundle on $X$ which we may consider as a complex concentrated in degree zero. Then, since $\pr_{X^{[n]}}\colon \Xi\to X^{[n]}$ is flat and finite of degree $n$, the object $E^{[n]}$ is a vector bundle (identified with a complex concentrated in degree zero) of $\rank E^{[n]}=n\cdot \rank E$. More generally, if $E\in \Coh(X)$ is a coherent sheaf, its image $E^{[n]}$ is again concentrated in degree zero; see \cite[Prop.\ 3]{Sca2}. Accordingly, we will also speak of \textit{tautological bundles} and \textit{tautological sheaves}.

\subsection{Derived McKay correspondence}\label{subsect:derivedMcKay}

Let $G$ be a finite group acting on a smooth quasi-projective variety $M$. A \textit{G-cluster} on $M$ is a zero-dimensional $G$-invariant closed subscheme $Z\subset M$ such that $\reg(Z)$ is given by the regular representation $\IC[G]$ of $G$. Every free $G$-orbit, viewed as a reduced subscheme, is a $G$-cluster. But there are also $G$-clusters whose support is a non-free $G$-orbit. These $G$-clusters are necessarily non-reduced. We denote by $\GHilb(M)$ the fine moduli space of $G$-clusters. The scheme $\GHilb$ can be reducible and we denote by $\Hilb^G(M)\subset \GHilb(M)$ the irreducible component containing all the points which correspond to free orbits. We call $\Hilb^G(M)$ the \textit{G-Hilbert scheme}. There is a morphism $\tau\colon \Hilb^G(M)\to M/G$, called the \textit{G-Hilbert--Chow morphism}, which sends $G$-clusters to the orbits on which they are supported.  
Let $\cZ\subset \Hilb^G(M)$ be the universal family of $G$-clusters. We have a commutative diagram
\[
\xymatrix{
 \cZ \ar^p[r]  \ar^{q}[d] & M\ar^{\pi}[d]   \\
 \Hilb^G(M) \ar^{\tau}[r] & M/G 
}
\]
where $p$ and $q$ are the projections and $\pi$ is the quotient morphism.

\begin{theorem}[\cite{BKR}]\label{thm:BKR}
Let $\omega_M$ be a locally trivial $G$-bundle, which means that for every $x\in M$ the stabiliser subgroup $G_x\le G$ acts trivially on the fibre $\omega_M(x)$. Furthermore, assume that 
\[\dim\bigl(\Hilb^G(M)\times_{M/G}\Hilb^G(M)\bigr)\le \dim M +1\]
where the fibre product is defined by the $G$-Hilbert--Chow morphism $\tau\colon \Hilb^G(M)\to M/G$. Then $\tau\colon \Hilb^G(M)\to M/G$ is a crepant resolution and 
\[
 \Phi:=p_*q^*\colon \D(\Hilb^G(M))\to \D_G(M)
\]
is an equivalence of triangulated categories. 
\end{theorem}
In this paper, we consider the case that $M=X^n$ is the cartesian power of a smooth quasi-projective surface $X$ and $G=\sym_n$ acts by permutation of the factors.
\begin{theorem}[\cite{Hai}]
There is an isomorphism $X^{[n]}\cong \Hilb^{\sym_n}(X^n)$ which identifies $\mu\colon X^{[n]}\to X^{(n)}$ and $\tau\colon \Hilb^{\sym_n}(X^n)\to X^{(n)}$. 
\end{theorem}
In particular, we get a commutative diagram
\begin{align}\label{diag:HC}
\xymatrix{
 \cZ \ar^p[r]  \ar^{q}[d] & X^n\ar^{\pi}[d]   \\
 X^{[n]} \ar^{\mu}[r] & X^{(n)} 
}
\end{align}
where $\cZ\subset X^{[n]}\times X^n$ is the universal family of $\sym_n$-clusters on $X^n$.

One can easily check that $\omega_{X^n}$ is locally trivial as a $\sym_n$-sheaf. Furthermore, by \cite{Briancon}, the Hilbert--Chow morphism $\mu\colon X^{[n]}\to X^{(n)}$ is semi-small which means that \[\dim(X^{[n]}\times_{X^{(n)}} X^{[n]})=\dim X^n=2n\,.\] 
Hence, the assumptions of \autoref{thm:BKR} are satisfied which gives 
\begin{cor}
 The functor $\Phi=p_*q^*\colon \D(X^{[n]})\to \D_{\sym_n}(X^n)$ is an equivalence.
\end{cor}

\begin{prop}\label{prop:Psiequi}
 The functor $\Psi:= p_*^{\sym_n}q^*\colon \D_{\sym_n}(X^n)\to \D(X^{[n]})$ is an equivalence too.
\end{prop}
\begin{proof}
The equivalence $\Phi\colon \D(X^{[n]})\to \D_{\sym_n}(X^n)$ is the equivariant Fourier--Mukai transform with kernel $\reg_\cZ\in \D_{\sym_n}(X^{[n]}\times X^n)$; see \cite{Plo, KSos} for details on equivariant Fourier--Mukai transforms. In general, a Fourier--Mukai transform is an equivalence if and only if the Fourier--Mukai transform with the same kernel in the reverse direction is an equivalence (but these two equivalences are usually not inverse to each other); see \cite[Rem.\ 7.7]{Huy}. In our case, the Fourier--Mukai transform with kernel $\reg_\cZ$ in the reverse direction is $\Psi\colon \D_{\sym_n}(X^n)\to \D(X^{[n]})$.  
\end{proof}

\subsection{Combinatorial notations}

Whenever we write intervals, they are meant as subsets of the integers. Concretely, for $a,b\in \IZ$ with $a\le b$, we have $[a,b]:=\{a,a+1,\dots, b\}\subset \IZ$. Furthermore, for $n\in \IN$ a positive integer, we set $[n]:=[1,n]=\{1,\dots, n\}$. For a subset, $I\subset [n]$, we write $\sym_I\le \sym_{[n]}$ for the subgroup of permutations fixing every element of $[n]\setminus I$. Clearly, $\sym_I\cong \sym_{|I|}$. We write $\alt_I$ for the alternating representation of $\sym_I$.

\subsection{Scala's theorem}\label{subsection:Scala}

Scala \cite{Sca1,Sca2} computed the image of tautological sheaves under the McKay correspondence $\Phi\colon \D(X^{[n]})\xrightarrow\cong \D_{\sym_n}(X^n)$. We describe his result in this subsection. 

Let $F\in \Coh(X)$ be a coherent sheaf on the surface $X$. Note that the projection $\pr_1\colon X^n\to X$ to the first factor is $\sym_{n-1}$-invariant, where $\sym_{n-1}\cong \sym_{[2,n]}$ acts by permutation of the last $n-1$ factors of $X^n$. Hence, the pull-back $\pr_1^*F$ carries a canonical $\sym_{n-1}$-linearisation.
We set 
\[
\CC^0_F:=\Ind_{\sym_{n-1}}^{\sym_n}\pr_1^*F\cong \bigoplus_{i=1}^n\pr_i^*F\in \Coh_{\sym_n}(X^n)\,.
\]
For $I\subset [n]$, we define the \textit{$I$-th partial diagonal} as the reduced subvariety
\[
\Delta_I:=\bigl\{(x_1,\dots, x_n)\mid x_i=x_j\,\forall i,j\in I  \bigr\}\subset X^n\,.
\]
We have an isomorphism $\Delta\cong X\times X^{n-|I|}$ where the first factor $X$ stands for the diagonal on the $I$-components. We denote by $\iota_I\colon X\times X^{n-|I|}\hookrightarrow X^n$ the embedding of the $I$-th partial diagonal and by $p_I\colon X\times X^{n-|I|}\to X$ the projection to the first factor. Then $\iota_{I*}p_I^* F$ carries a canonical $\sym_I\times \sym_{[n]\setminus I}$-linearisation and we set
\[
 F_I:=\iota_{I*}p_I^* F\otimes \alt_I\in \Coh_{\sym_I\times \sym_{[n]\setminus I}}(X^n)
\]
and, for $1\le p\le n-1$, 
\[
 \CC^p_F:=\Ind_{\sym_{p+1}\times\sym_{n-p-1}}^{\sym_n}F_{[p+1]}\cong \bigoplus_{I\subset [n],\, |I|=p+1}F_I\in \Coh_{\sym_n}(X^n)\,. 
\]
For $I\subset J$, we have $\Delta_J\subset \Delta_I$. Hence, there are morphisms $F_I\to F_J$ given by restriction of local sections. Using appropriately alternating sums of these morphisms, we get 
$\sym_n$-equivariant differentials $d^p\colon \CC^p_F\to \CC^{p+1}_F$; see \cite[Rem.\ 2.2.2]{Sca1} for details.
Hence we have defined a complex $\CC^\bullet_F\in \D_{\sym_n}(X^n)$.
\begin{theorem}[\cite{Sca1, Sca2}]\label{thm:Scala}
For $F\in \Coh(X)$, there are functorial isomorphisms 
 $\Phi(F^{[n]})\cong \CC^\bullet_F$ in $\D_{\sym_n}(X^n)$.
\end{theorem}
\begin{remark}\label{rem:Postnikov}
Note that the definition of the $\CC^p_F$ still makes perfect sense as an object in $\D_{\sym_n}(X^n)$ if we replace the sheaf $F\in \Coh(X)$ by a complex $F\in \D(X)$. Also, for $F\in \D(X)$, there are still morphisms $d^p\colon \CC^p_F\to \CC^{p+1}_F$ in $\D_{\sym_n}(X^n)$. This allows us to define $\CC^\bullet_F$, for $F\in \D(X)$, as a Postnikov system; for this notion which, roughly speaking, generalises the notion of complexes from abelian to triangulated categories, see e.g.\ \cite{Orlsurvey}. 
With this definition of $\CC^\bullet_F$, the statement of \autoref{thm:Scala} remains true for $F\in \D(X)$ instead of $F\in \Coh(X)$. Another way to phrase this is that there is an isomorphism of functors
\[
 \Phi\circ(\_)^{[n]}\cong \CC^\bullet_{(\_)}\colon \D(X)\to\D_{\sym_n}(X^n)
\]
where $\CC^\bullet_{(\_)}=\FM_{\cK^\bullet}$ is the Fourier--Mukai transform along the $\sym_n$-equivariant complex
\[
\cK^\bullet=\bigl(0\to\bigoplus_{i=1}^n \reg_{D_i} \to \bigoplus_{|I|=2}\reg_{D_I}\to \dots \to \reg_{D_{[n]}}\to 0\bigr)  
\]
on $X\times X^n$. Here, the $D_I$ are the reduced subvarieties given by $D_I:=\cap_{i\in I} D_i$ where
\[
 X^n\cong D_i=\bigl\{(x,x_1,\dots,x_n)\mid x=x_i \bigr\}\subset X\times X^n\,.
\]
\end{remark}


\section{Tautological bundles under the derived McKay correspondence}

In this section, we proof \autoref{thm:intro1}. For this purpose, we introduce various families related to tautological bundles over the Hilbert scheme and discuss their geometry.

\subsection{Various universal families and their geometry}

We define $\Xi(n,k)\subset X^{[n]}\times X^k$ as the $k$-fold fibre product 
\[
 \Xi(n,k):=\Xi\times_{X^{[n]}} \Xi\times_{X^{[n]}}\dots \times_{ X^{[n]}} \Xi\,.
\]
It is the reduced (see \cite[Sect.\ 1.4]{Sca1}) subvariety of $X^{[n]}\times X^k$ given by
\[
 \Xi(n,k)=\bigl\{(\xi, x_1,\dots, x_k)\mid x_i\in \xi \, \forall \, i=1,\dots,k  \bigr\}\,.
\]
We denote the projections by $e_k\colon \Xi(n,k)\to X^{[n]}$ and $f_k\colon \Xi(n,k)\to X^{k}$.
\begin{lemma}\label{lem:wedgetaut}
For $E\in \D(X)$ and $k\in \IN$, we have natural isomorphisms
\[\wedge^k(E^{[n]})\cong e_{k*}^{\sym_k}f_k^*(E^{\boxtimes k}\otimes \alt_k)\,.\]
\end{lemma}
\begin{proof}
By flat base change along the cartesian diagram
\[
\xymatrix{
 \Xi(n,k) \ar^{\delta'}[r]  \ar_{e_k}[d]\ar@/^6mm/^{f_k}[rr] & \Xi^k\ar_{\pr_{X^{[n]}}^k}[d] \ar_{\pr_{X}^k}[r] & X^k  \\
 X^{[n]} \ar_{\delta}[r] & (X^{[n]})^{k} 
}
\]
where $\delta$ is the diagonal embedding, we get
\begin{align*}
 (E^{[n]})^{\otimes k}\cong \delta^*((E^{[n]})^{\boxtimes k})\cong \delta^*\pr_{X^{[n]}*}^k\pr_X^{k*}(E^{\boxtimes k})\cong e_{k*}f_k^*(E^{\boxtimes k})\,. 
\end{align*}
The isomorphism $(E^{[n]})^{\otimes k}\cong e_{k*}f_k^*(E^{\boxtimes k})$ is $\sym_k$-equivariant with $\sym_k$ acting on both sides by permutation of the tensor factors. Hence, taking $\sym_k$-anti-invariants on both sides gives the assertion. 
\end{proof}

In \cite{Hai}, a morphism $\Hilb^{\sym_n}(X^n)\to X^{[n]}$ (which is then shown to be an isomorphism) is constructed as follows. Let $\cZ\subset \Hilb^{\sym_n}(X^n)\times X^n$ be the universal family of $\sym_n$-clusters. One shows that $\id\times \pr_1(\cZ)\subset \Hilb^{\sym_n}\times X$ is flat of degree $n$ over $\Hilb^{\sym_n}(X^n)$ which results in the classifying morphism $\Hilb^{\sym_n}(X^n)\to X^{[n]}$. 

A posteriori, once the inverse morphism $X^{[n]}\to \Hilb^{\sym_n}(X^n)$ and hence the identification $X^{[n]}\cong \Hilb^{\sym_n}(X^n)$ is established, one can interpret this as follows. Let $\cZ\subset X^{[n]}\times X^n$ be the universal family of $\sym_n$-clusters and $\Xi\subset X^{[n]}\times X$ be the universal family of length $n$ subschemes of $X$. Then $\id\times \pr_1(\cZ)=\Xi$ and we have $\Xi\cong \cZ/\sym_{n-1}$ with $\id\times \pr_1$ being the quotient morphism.

More generally, let $k\in [n]$ and denote by $\pr_{[k]}\colon X^n\to X^k$ the projection to the first $k$ factors. We consider the closed subvariety 
\[
 \Xi\binom nk:=\id\times \pr_{[k]}(\cZ)=\{(\xi,x_1,\dots,x_k)\mid \mu(\xi)\ge x_1+\dots+x_k\}\subset X^{[n]}\times X^k
\]
where the inequality $\mu(\xi)\ge x_1+\dots +x_k$ means that every point occurs in the left-hand side with at least the same multiplicity as in the right-hand side. 
Note that $\Xi\binom n1=\Xi$.
We denote the restriction of the projection $\id\times \pr_{[k]}$ by $q_{k}\colon\cZ\to \Xi\binom nk$. 
\begin{prop}\label{prop:qkquotient}
The projection $q_k\colon \cZ\to \Xi\binom nk$ is the $\sym_{[k+1,n]}$-quotient of $\cZ$.  
\end{prop}
Since $\cZ$ is the universal family of $\sym_n$-clusters, the proposition follows from this
\begin{lemma}
Let $J\subset \IC[x_1,y_1\dots,x_n,y_n]$ be the vanishing ideal of a $\sym_n$-cluster on $(\IA^2)^n$. Then, for $k\le n$, the inclusion $\IC[x_1,y_1,\dots, x_k,y_k]\hookrightarrow \IC[x_1,y_1,\dots,x_n,y_n]$ induces an isomorphism  
\[
\IC[x_1,y_1,\dots,x_k,y_k]/(J\cap \IC[x_1,y_1,\dots,x_k,y_k])\xrightarrow\cong (\IC[x_1,y_1,\dots,x_n,y_n]/J)^{\sym_{n-k}}\,. 
\]
\end{lemma}
\begin{proof}
For $k=1$, the proof can be found in \cite[Sect.\ 4]{Haigeometry} and it works the same for arbitrary $k$. We reproduce the argument for completeness sake. 
For a finite set of variables $\{x_i,y_i\}_{i\in I}$ and non-negative integers $h,k\in\IN$, we consider the $\sym_I$-invariant power sum polynomial 
\[
 p_{h,k}(\{x_i,y_i\}_{i\in I})=\sum_{i\in I} x_i^h y_i^k\,.
\]
By definition of a $\sym_n$-cluster, $\IC[x_1,y_1,\dots,x_n,y_n]/J$ is given by the regular representation. In particular, its $\sym_n$-invariants are one-dimensional. Hence, every power sum polynomial $p_{h,k}(x_1,y_1,\dots,x_n,y_n)$ is congruent to a constant $c_{h,k}$ modulo $J$. By a theorem of Weyl, $\IC[x_1,y_1,\dots,x_n,y_n]^{\sym_{n-k}}$ is generated by the power sums $p_{h,k}(x_{k+1},y_{k+1},\dots,x_n,y_n)$, as a $\IC[x_1,y_1,\dots, x_k,y_k]$-algebra; see e.g.\ \cite[Thm.\ 1.2]{Dalbec}. We have
\begin{align*}
p_{h,k}(x_{k+1}, y_{k+1},\dots, x_n,y_n)&=p_{h,k}(x_1,y_1,\dots,x_n,y_n)- p_{h,k}(x_1,y_1,\dots,x_k,y_k)\\&\equiv 
c_{h,k}- p_{h,k}(x_1,y_1,\dots,x_k,y_k)\quad\mod J\,.
\end{align*}
This shows that the inclusion     
\[
\IC[x_1,y_1,\dots,x_k,y_k]/(J\cap \IC[x_1,y_1,\dots,x_k,y_k])\to (\IC[x_1,y_1,\dots,x_n,y_n]/J)^{\sym_{n-k}}
\]
is also surjective.
\end{proof}

\subsection{Tautological objects under the derived McKay correspondence}\label{sec:tautBKRH}

\begin{definition}\label{def:CW}
We consider the functor 
\[
 \CC:=\Ind_{\sym_{n-1}}^{\sym_n} \pr_1^*\colon \D(X)\to \D_{\sym_n}(X^n)\,.
\]
Furthermore, for $1\le k \le n$ and $F\in \D(X)$, we set
\[
 \WW^k(F):=\Ind_{\sym_k\times \sym_{n-k}}^{\sym_n}(\pr_{[k]}^*(L^{\boxtimes k}\otimes \alt_k))\cong \bigoplus_{\substack{I\subset [n]\\ |I|=k}} \pr_I^* (L^{\boxtimes k})\otimes \alt_I\,\in \D_{\sym_n}(X^n)\,.
\]
Note that $\CC(F)\cong\CC^0_F\cong  \WW^1(F)$. We also set $\WW^0(F):=\reg_{X^n}$ (the structure sheaf equipped with the canonical linearisation). 
\end{definition}

We denote by $e_k'\colon \Xi\binom nk\to X^{[n]}$ and $f_k'\colon \Xi\binom nk\to X^{[n]}$ the projections. 

\begin{prop}\label{prop:Psi}
There are isomorphisms, functorial in $F\in \D(X)$, 
\[
 \Psi(\WW^k(F))\cong  e'^{\sym_k}_{k*}f'^*_k(F^{\boxtimes k}\otimes \alt_k) \,.
\]
\end{prop}
\begin{proof}
We consider the commutative diagram
\[
\begin{xy}
\xymatrix{
 \cZ \ar^p[r]  \ar^{q_k}[d] \ar@/_6mm/_q[dd]& X^n\ar^{\pr_{[k]}}[d]   \\
 \Xi\binom nk \ar^{f'_k}[r]\ar^{e'_k}[d] & X^k  \\
X^{[n]} & .
}
\end{xy} 
\]
We get
\begin{align*}
 \Psi(\WW^k(F))&\cong q_*^{\sym_n}p^*\Ind_{\sym_k\times \sym_{n-k}}^{\sym_n} \pr_{[k]}^*(F^{\boxtimes k}\otimes \alt_k)\\
&\overset{\eqref{Indcommutes}}{\cong} (\_)^{\sym_n}\Ind_{\sym_k\times \sym_{n-k}}^{\sym_n}e'_{k*}q_{k*}q_k^*f_k'^*(F^{\boxtimes k}\otimes \alt_k)\\
&\overset{\ref{lem:globalFrob}}{\cong} (\_)^{\sym_k\times \sym_{n-k}} e'_{k*}q_{k*}q_k^*f_k'^*(F^{\boxtimes k}\otimes \alt_k)\\
&\overset{\eqref{invacommutes}}{\cong}  e'^{\sym_k}_{k*}q^{\sym_{n-k}}_{k*}q_k^*f_k'^*(F^{\boxtimes k}\otimes \alt_k)\\
&\overset{\ref{lem:quotientff}+ \ref{prop:qkquotient}}{\cong}  e'^{\sym_k}_{k*}f_k'^*(F^{\boxtimes k}\otimes \alt_k)\,.\qedhere
\end{align*}
 \end{proof}

\begin{theorem}\label{thm:PsiCC}
There is an isomorphism of functors $(\_)^{[n]}\cong \Psi\circ \CC$.
\end{theorem} 
\begin{proof}
Since $\Xi\binom n1=\Xi$, this is the case $k=1$ of the previous proposition.
\end{proof}

Using the language of equivariant Fourier--Mukai transforms, we can write \autoref{lem:wedgetaut} and \autoref{prop:Psi} as 
\[
\wedge^k(F^{[n]})\cong \FM_{\reg_{\Xi(n,k)}}(F^{\boxtimes k}\otimes  \alt_k)^{\sym_k} \quad\text{and} \quad \Psi(\WW^k(F))\cong \FM_{\reg_{\Xi\binom nk}}(F^{\boxtimes k}\otimes  \alt_k)^{\sym_k}\,,  
\]
respectively. Since $\Xi\binom nk\subset \Xi(n,k)$ we get a morphism $\reg_{\Xi(n,k)}\to \reg_{\Xi\binom nk}$ between the Fourier--Mukai kernels given by restriction of local regular functions.
\begin{lemma}\label{lem:antiinva}
The restriction map $\reg_{\Xi(n,k)}\to \reg_{\Xi\binom nk}$ induces an isomorphism on the level of anti-invariants
\[
e_{k*}(\reg_{\Xi(n,k)}\otimes\alt_k)^{\sym_k}\cong e'_{k*}(\reg_{\Xi\binom nk}\otimes\alt_k)^{\sym_k}\,. 
\]
\end{lemma}
\begin{proof}
The universal family $\cZ\subset X^{[n]}\times X^n$ is irreducible of dimension $2n$; see \cite[Prop.\ 3.3.2]{Hai}. Thus, the same holds for the quotient $\Xi\binom nk\cong\cZ/\sym_{n-k}$. Hence, $\Xi\binom nk$ is an irreducible component of $\Xi(n,k)$ as the latter is also finite over $X^{[n]}$ and consequently of dimension $2n$. Since both, $\Xi\binom nk$ and $\Xi(n,k)$, are reduced, it is sufficient to show that every $\sym_k$-anti-invariant function $f\in \reg_{\Xi(n,k)}$ vanishes on the complement of $\Xi\binom nk$. 

Every point $p\in\Xi(n,k)\setminus \Xi\binom nk$ is of the form $p=(\xi,x_1,\dots,x_k)$ where $x_i=x_j$ for at least one pair $1\le i<j\le n$. Set $\tau=(i\,\,j)\in \sym_k$. Then $\tau(p)=p$. It follows that every $\sym_k$-anti-invariant function satisfies $f(p)=-f(p)$, hence $f(p)=0$. 
\end{proof}

\begin{cor}\label{cor:wedgeFM}For $L\in \Pic X$ and $0\le k\le n$, we have
 \[\wedge^k(L^{[n]})\cong \FM_{\reg_{\Xi\binom nk}}(L^{\boxtimes k}\otimes \alt_k)^{\sym_k}\cong {e'_k}_{*}^{\sym_k}{f'_k}^{*}(L^{\boxtimes k}\otimes \alt_k)\,.\] 
\end{cor}
\begin{proof}
 This follows from \autoref{lem:wedgetaut} together with \autoref{lem:antiinva}. 
\end{proof}

\begin{theorem}\label{thm:wedgeline}
For $L\in \Pic(X)$ and $0\le k\le n$, we have $\Psi(\WW^k(L))\cong \wedge ^k(L^{[n]})$.  
\end{theorem}
\begin{proof}
 This follows from \autoref{prop:Psi} together with \autoref{cor:wedgeFM}. 
\end{proof}

\begin{remark}
 The $k=0$ case of \autoref{thm:wedgeline} says that $\Psi(\reg_{X^n})\cong \reg_{X^{[n]}}$. The $k=n$ case gives 
 $\Psi(\reg_{X^n}\otimes \alt_n)\cong \det \reg^{[n]}\cong \reg(-B/2)$ where $B\subset X^{[n]}$ is the effective divisor parametrising non-reduced length $n$ subschemes of $X$. 
\end{remark}

\begin{remark}\label{rem:DM}
Using the commutative diagram \eqref{diag:HC}, one can easily show that the functor $\Psi\colon \D_{\sym_n}(X^n)\to \D(X^{[n]})$ is $\reg_{X^{(n)}}$-linear, which means that, for $A\in \Dperf(X^{(n)})$ and $B\in \D_{\sym_n}(X^n)$, we have functorial isomorphisms $\Psi(\pi^*A\otimes B)\cong\mu^*(A)\otimes \Psi(B)$. Hence, \autoref{thm:PsiCC} and \autoref{thm:wedgeline} actually give descriptions for the image of objects under the derived McKay correspondence for a much larger class than only tautological objects and wedge powers of tautological bundles. In particular, let $M\in \Pic X$ be a line bundle. Then the equivariant line bundle $M^{\boxtimes n}\in \Pic_{\sym_n}(X^n)$ descends to the line bundle $M^{(n)}:=\pi_*^{\sym_n}(M^{\boxtimes n})\in \Pic X^{(n)}$. Set $\cD_M:=\mu^*(M^{(n)})$. Then, for $L\in \Pic X$ and $F\in \D(X)$, we get
\[
 \Psi(\CC(F)\otimes M^{\boxtimes n})\cong F^{[n]}\otimes \cD_M\quad, \quad \Psi(\WW^k(L)\otimes M^{\boxtimes n})\cong \wedge^k(L^{[n]})\otimes \cD_M\,. 
\]
\end{remark}

\begin{remark}\label{rem:Heisenberginterpretation}
For $Y$ a smooth projective variety, the equivariant derived category $\D_{\sym_n}(Y^n)$ has some interesting features. For example, if $\D(Y)$ carries a (full, strong) exceptional collection one can construct a (full, strong) exceptional collection on $\D_{\sym_n}(Y^n)$. The same holds for semi-orthogonal decompositions and tilting bundles; see \cite[Sect.\ 4]{KSos}. Furthermore, there is always an action of a Heisenberg algebra on the category $\D_{\sym_n}(Y^n)$; see \cite{CL} and \cite{KrugHeisenberg}. If $Y=X$ is a surface, there is the McKay equivalence $\D(X^{[n]})\cong \D_{\sym_n}(X^n)$. Hence, abstractly, we know that the derived category $\D(X^{[n]})$ carries all the above features too. However, it would be interesting to understand the constructions (of exceptional sequences, the Heisenberg action, etc.)\ concretely in geometric terms on the Hilbert scheme. \autoref{thm:wedgeline} can be seen as a step into that direction as the objects $\WW^k(E)$ play an important role in all these constructions on $\D_{\sym_n}(X^n)$.  
\end{remark}

\begin{remark}
 For $n=2$, the autoequivalence $\Phi\Psi$ of $\D_{\sym_n}(X^n)$ can be computed as 
\begin{align}\label{PhiPsi}
 \Phi\Psi\cong \MM_{\alt} \TT^{-1}_{\delta_*} \MM_\alt\,;
\end{align}
see \cite[Sect.\ 4.6]{KPScyclic}.
Here, $\MM_\alt\in \Aut(\D_{\sym_2}(X^2))$ is the tensor product by the non-trivial character of $\sym_2$ and $\TT_{\delta_*}\in \Aut(\D_{\sym_2}(X^2))$ is the twist along the push-forward along the diagonal $\delta_*\colon\D(X)\to \D_{\sym_2}(X^2)$ which is a spherical functor; see \cite{Kru3}. Concretely, 
\[
 \TT^{-1}_{\delta_*}=\cone\bigl(\id \xrightarrow\eps \delta_*\delta^*(\_)^{\sym_2}\bigr)
\]
as a cone of Fourier--Mukai transforms where $\eps$ is the unit of adjunction which, in this case, is given by restriction of sections to the diagonal followed by projection to the invariants. Using this, one can compute directly that, for $n=2$, we have $\Phi\Psi(\CC(F))\cong \CC^\bullet_F$ which explains the difference between \autoref{thm:Scala} and \autoref{thm:PsiCC}. Conversely, for $n>2$, the difference between $\CC^\bullet_F$ and $\CC(F)$ allows to guess how the autoequivalence $\Phi\Psi\in \Aut(\D_{\sym_n}(X^n))$ could look like in general. Namely, one can hope that \eqref{PhiPsi} still holds with $\delta_*$ replaced by a spherical functor $\D_{\sym_{n-2}}(X\times X^{n-2})\to \D_{\sym_n}(X^n)$ whose image is supported on the big diagonal in $X^n$; see \cite[Sect.\ 5.9]{Kru4} for some speculation on these kind of spherical functors.       
\end{remark}

\begin{remark}\label{rem:Kummer}
Let $X=A$ be an abelian surface. Then there is the \textit{generalised Kummer variety}, which is the subvariety 
$K_{n-1}A\subset A^{[n]}$ parametrising length $n$ subschemes of $A$ whose weighted support adds up to $0\in A$. 
It is smooth of dimension $2(n-1)$. We also consider the $\sym_n$-invariant subvariety
\[
 N_{n-1}A=\{(a_1,\dots, a_n)\mid a_1+\dots+a_n=0\}\subset A^n\,.
\]
We denote the embeddings by $i\colon K_{n-1}A\hookrightarrow A^{[n]}$ and $j\colon N_{n-1}A\hookrightarrow A^n$, respectively.
The equivalence $\Psi\colon \D_{\sym_n}(A^n)\xrightarrow\cong \D(A^{[n]})$ restricts to a functor $\Psi_K\colon \D_{\sym_n}(N_{n-1}A)\xrightarrow\cong \D(K_{n-1}A)$ which is again an equivalence and satisfies 
\begin{align}\label{eq:Psicommutes}
\Psi_K j^*\cong i^*\Psi\,.
\end{align}
This follows from \cite[Lem.\ 6.1 \& Prop.\ 6.2]{Chenflops}, see also \cite[Lem.\ 6.2]{Mea} for details of the analogous argument for $\Phi\colon \D(A^{[n]})\to \D_{\sym_n}(A^n)$ instead of $\Psi$. Again, we get a tautological functor $K_{n-1}\colon \D(A)\to \D(K_{n-1}A)$ by means of the Fourier--Mukai transform along the universal family of the generalised Kummer variety. It satisfies $K_{n-1}\cong i^*(\_)^{[n]}$. Hence, using    
\eqref{eq:Psicommutes}, \autoref{thm:PsiCC}, and \autoref{thm:wedgeline}, we get an isomorphism of functors $K_{n-1}\cong \Psi_K j^* \CC$ and  
\[
 \wedge^k(K_{n-1}(L))\cong \Psi_K(j^*\WW^k(L))\quad\text{for $L\in \Pic A$ and $0\le k\le n-1$.}
\]
\end{remark}

\begin{remark}
Let $F\colon \D(X)\to \D(X^{[n]})$ be the Fourier--Mukai transform along the universal ideal sheaf $\cI_\Xi$ (recall that, in contrast, the tautological functor $(\_)^{[n]}$ is the Fourier--Mukai transform along $\reg_\Xi$). If $X$ is a K3 surface, $F$ is a $\IP^{n-1}$-functor which means in particular that the composition with its right-adjoint is given by
\[
 F^RF\cong \id_{\D(X)}\oplus [-2]\oplus [-4]\oplus\dots\oplus [-2(n-1)]\,;
\]
see \cite{Add} where this result is proved using incidence subschemes of the products $X^{[n]}\times X^{[n+1]}$.
Shorter proofs using the derived McKay correspondence and Scala's result \autoref{thm:Scala} were given in \cite{MMdefK3,Mea}. These proofs can be simplified further using \autoref{thm:PsiCC} instead of \autoref{thm:Scala}. Furthermore, \autoref{thm:PsiCC} explains the occurrence of the 'truncated universal ideal functors' of \cite[Sect.\ 5]{KSos}.
\end{remark}

\section{Extension groups}

Using the results from the previous section, we can derive various formula for the cohomologies and extension groups of tautological objects.
For this we need the following general formula for the graded Hom-spaces in the equivariant derived category $\D_{\sym_n}(X^n)$ between the objects given by the constructions of \autoref{def:CW}.

\begin{prop}\label{prop:WExt}
Let $E,F\in \D(X)$ and $0\le e,f\le n$. Then 
\begin{align*}
&\Hom^*_G(\WW^e(E), \WW^f(F))\\\cong &\bigoplus_{i=\max \{0,e+f-n\}}^{\min\{e,f\}}S^i\Hom^*(E,F)\otimes \wedge^{e-i}\Ho^*(E^\vee)\otimes \wedge^{f-i}\Ho^*(F)\otimes S^{n+i-e-f}\Ho^*(\reg_X)\,. 
\end{align*}
\end{prop}

\begin{proof}
The proof is an application of \autoref{rem:Danila} together with \eqref{GinvaHom}.
Recall that, for the underlying non-equivariant objects, we have
\[
 \WW^e(E)\cong \bigoplus_{I\subset [n],\, |I|=e} \pr_I^*E^{\boxtimes e}\,;
\]
see \autoref{def:CW}.
Hence,  
\[
 \Hom^*(\WW^e(E), \WW^f(F))\cong\bigoplus_{(I,J)\in \cI}\Hom^*(\pr_I^*E^{\boxtimes e}, \pr_J^*F^{\boxtimes f})
\]
where $\cI=\{(I,J)\mid I,J\subset [n], |I|=e, |J|=f\}$. The linearisations of $\WW^e(E)$ and $\WW^f(F)$ induce on $\cI$ the action $\sigma\cdot(I,J)=(\sigma(I), \sigma(J))$. A set of representatives of the orbits under this action is given by 
the pairs of the form \[P_i:=([e], [i]\cup [e+1, e+f-i])\quad\text{for} \quad i\in [\max\{0, e+f-n\}, \min\{e,f\}]\,.\] The stabiliser of $P_i$ is $G_i:=\sym_i\times \sym_{e-i}\times \sym_{f-i}\times \sym_{n-e-f+i}\le \sym_n$. By the K\"unneth formula, we see that $\Hom^*(\pr_{[e]}^*E^{\boxtimes e}, \pr_{[i]\cup [e+1, e+f-i]}^*F^{\boxtimes f})$, as a $G_i$-representation, is given by
\begin{align*}
 &\Hom^*\bigl(\pr_{[e]}^*(E^{\boxtimes e}\otimes \alt_e), \pr_{[i]\cup [e+1, e+f-i]}^*(F^{\boxtimes f}\otimes \alt_f)\bigr)\\\cong &\Hom^*(E,F)^{\otimes i}\otimes (\Ho^*(E^\vee)^{\otimes e-i}\otimes \alt_{e-i}) \otimes (\Ho^*(F)^{\otimes f-i}\otimes \alt_{f-i})\otimes\Ho^*(\reg_X)^{\otimes n+i-e-f}\,. 
\end{align*}
Hence, its $G_i$-invariants are given by
\[S^i\Hom^*(E,F)\otimes \wedge^{e-i}\Ho^*(E^\vee)\otimes \wedge^{f-i}\Ho^*(F)\otimes S^{n+i-e-f}\Ho^*(\reg_X)\]
and \autoref{rem:Danila} together with \eqref{GinvaHom} give the result.  
\end{proof}

\begin{cor}\label{cor:Extformulae} For $E,F\in \D(X)$, $K,L\in \Pic(X)$ and $k,\ell\in [n]$, we have natural isomorphisms
\begin{align}
\Ho^*(F^{[n]})&\cong \Ho^*(F)\otimes S^{n-1}\Ho^*(\reg_X)\,,\label{cohF}
\\
\Ho^*((E^{[n]})^\vee)&\cong \Ho^*(E^\vee)\otimes S^{n-1}\Ho^*(\reg_X)\,,\label{cohEvee}
\\
\Hom^*(E^{[n]},F^{[n]})&\cong \begin{aligned}& \Hom^*(E,F) \otimes S^{n-1}\Ho^*(\reg_X)\label{ExtEF}
\\&\oplus \Ho^*(E^\vee)\otimes \Ho^*(F)\otimes S^{n-2}\Ho^*(\reg_X)\,, 
\end{aligned}
\\
\Ho^*(\wedge^kL^{[n]})&\cong \wedge^k\Ho^*(L)\otimes S^{n-k}\Ho^*(\reg_X)\,,\label{cohwedge}
\\
\Hom^*(E^{[n]}, \wedge^k L^{[n]})&\cong\begin{aligned}& \Hom^*(E,L)\otimes \wedge^{k-1}\Ho^*(L) \otimes S^{n-k}\Ho^*(\reg_X)\label{ExtEwedge}
\\&\oplus \Ho^*(E^\vee)\otimes \wedge^k\Ho^*(L)\otimes S^{n-k-1}\Ho^*(\reg_X)\,. 
\end{aligned}
\\
\Hom^*(\wedge^k L^{[n]}, F^{[n]})&\cong\begin{aligned}& \Hom^*(L,F)\otimes \wedge^{k-1}\Ho^*(L^\vee) \otimes S^{n-k}\Ho^*(\reg_X) \label{ExtwedgeF}
\\&\oplus \wedge^k\Ho^*(L^\vee)\otimes \Ho^*(F)\otimes S^{n-k-1}\Ho^*(\reg_X)\,,
\end{aligned}
\end{align}
\begin{align}
\begin{aligned}\label{Extwedgewedge}
&\Hom^*(\wedge^k K^{[n]}, \wedge^\ell L^{[n]})\\\cong& \bigoplus_{i=\max \{0,k+\ell-n\}}^{\min\{k,\ell\}}S^i\Hom^*(K,L)\otimes \wedge^{k-i}\Ho^*(K^\vee)\otimes \wedge^{\ell-i}\Ho^*(L)\otimes S^{n+i-k-\ell}\Ho^*(\reg_X)\,.
\end{aligned}
\end{align}
\end{cor}
\begin{proof}
Since $\Psi\colon \D_{\sym_n}(X^n)\to \D(X^{[n]})$ is an equivalence, we have
\[\Hom^*(\Psi(\cE),\Psi(\cF))\cong \Hom^*_{\sym_n}(\cE, \cF)\quad\text{for every $\cE, \cF\in \D_{\sym_n}(X^n)$.}\]
Using this, all the formulae follow from \autoref{thm:PsiCC}, \autoref{thm:wedgeline}, and \autoref{prop:WExt}. Concretely, to obtain the formulae as special cases of \autoref{prop:WExt} we have to set
\begin{align*}
&\text{for \eqref{cohF}:}\quad E=\reg_X, e=0, f=1\,, \quad \text{for \eqref{cohEvee}:}\quad F=\reg_X, e=1, f=0 \,, \quad \text{for \eqref{ExtEF}:}\quad e=1, f=1\,,\\
&\text{for \eqref{cohwedge}:}\quad E=\reg_X, F=L, e=0, f=k\,, \quad \text{for \eqref{ExtEwedge}:}\quad F=L, e=1, f=k \,,\\
&\text{for \eqref{ExtwedgeF}:}\quad E=L, e=k, f=1\,,
\quad\text{for \eqref{Extwedgewedge}:}\quad E=L, F=K, e=\ell, f=k\,.\qedhere
\end{align*}

\end{proof}

\begin{remark}
Using \autoref{rem:DM}, one can easily generalise \autoref{cor:Extformulae} to formulae involving tensor products by the natural line bundles $\cD_M$. For example, \eqref{ExtEwedge} becomes    
\begin{align*}
&\Hom^*(E^{[n]}\otimes \cD_M, \wedge^k (L^{[n]})\otimes \cD_N)\\\cong& \Hom^*(E\otimes M,L\otimes N)\otimes \wedge^{k-1}\Hom^*(M, L\otimes N) \otimes S^{n-k}\Hom^*(M,N)
\\&\oplus \Hom^*(E\otimes M,N)\otimes \wedge^k\Hom^*(M,L\otimes N)\otimes S^{n-k-1}\Hom^*(M, N)\,. 
\end{align*}
\end{remark}
\begin{remark}
Using \autoref{rem:Kummer}, one can derive formulas, similar to those of \autoref{cor:Extformulae}, for the cohomology and extension groups of bundles on the generalised Kummer varieties; compare \cite[Thm.\ 6.9]{Mea}. 
\end{remark}

Formula \eqref{cohF} has been proved in different levels of generality in \cite{Dan, Sca1, Sca2}. The formulae \eqref{cohEvee} and \eqref{ExtEF} are proved in \cite{KruExt} and formula \eqref{cohwedge} is proved in \cite{Sca1}. 
To the best of the authors knowledge, \eqref{ExtEwedge}, \eqref{ExtwedgeF}, and \eqref{Extwedgewedge} are new. They generalise and strengthen the results of \cite{WangZhouTaut} on the Euler bicharacteristic of wedge powers of tautological bundles associated to line bundles on the surface as we see in the following.

\begin{definition}\label{sdefin}
 For $k\in \IN$ and $\chi\in \IZ$ we define
\[
\lambda^k\chi:=\frac 1{k!}\chi(\chi-1)\cdots(\chi-k+1)\quad,\quad s^k\chi:=\frac 1{k!}(\chi+k-1)(\chi+k-2)\cdots\chi\,. 
\]
\end{definition}
If $\chi\in \IN$, we have $\lambda^k\chi=\binom\chi k$ and $s^k\chi=\binom{\chi+k-1} k$. It follows that, for a finite dimensional graded vector space $V^*$, we have
\[
 \chi(\wedge^kV^*)=\lambda^k\chi(V^*)\quad, \quad \chi(S^kV^*)=s^k\chi(V^*)
\]
where $\chi(V^*)=\sum_{i\in \IZ}(-1)^i\dim(V^i)$ is the Euler characteristic and the wedge and symmetric powers are formed in the graded sense, i.e.\ they take into account signs coming from the grading.

\begin{cor}\label{cor:bichar} If $X$ is projective, we have
 \[
  \chi(\wedge^k K^{[n]}, \wedge^\ell L^{[n]})=\sum_{i=\max \{0,k+\ell-n\}}^{\min\{k,\ell\}}s^i\chi(K,L)\cdot \lambda^{k-i}\chi(K^\vee)\cdot \lambda^{\ell-i}\chi(L)\cdot s^{n+i-k-\ell}\chi(\reg_X)
 \]
 for every $L,K\in \Pic X$ and $k,\ell,n\in \IN$ non-negative integers.
In terms of generating functions this can be expressed equivalently as  
\begin{align}\label{KLcounting}
\sum_{n=0}^\infty \chi(\Lambda_{-v}K^{[n]}, \Lambda_{-u} L^{[n]})Q^n= \exp\left(\sum_{r=1}^\infty \chi(\Lambda_{-v^r} K, \Lambda_{-u^r} L) \frac{Q^r}r    \right) 
\end{align}
where for a vector bundle $E$ of rank $r$ and a formal parameter $t$, we use the convention $\Lambda_{t}E=\sum_{i=0}^r(\wedge^i E)t^i$ as a sum in the Grothendieck group.
\end{cor}
\begin{proof}
The first formula follows directly from \eqref{Extwedgewedge}. The equivalence of the two formulae is shown in \autoref{Appendix}. 
\end{proof}
 
We will further discuss the formula for the Euler bicharacteristic in \autoref{sec:further}.

\section{Tensor products and their Euler characteristic}

\subsection{Invariants of tensor products under the McKay correspondence}
For the computation of extension groups of sheaves or objects $F\in \D(X^{[n]})$ on the Hilbert scheme, we can either use a description of $\Phi(F)$ or one of $\Psi^{-1}(F)$ and we have seen that the latter is often more convenient. For the computation of the cohomology of tensor products, however, it can be useful to have both descriptions at the same time due to the following
\begin{prop}\label{prop:relativeFMtensor} For $E\in \D(X^{[n]})$ and $F\in\D_{\sym_n}(X^n)$ there are functorial isomorphisms  
 \[
  \mu_*(E\otimes \Psi(F))\cong \pi^{\sym_n}_*(\Phi(E)\otimes F)\,.
 \]
\end{prop}
\begin{proof}
 This follows from the commutativity of the diagram 
 \[
\xymatrix{
 \cZ \ar^p[r]  \ar^{q}[d] & X^n\ar^{\pi}[d]   \\
 X^{[n]} \ar^{\mu}[r] & X^{(n)} 
}
\]
together with the projection formula:
\begin{align*}
\mu_*(E\otimes \Psi(F))\cong \mu_*(E\otimes q_*^{\sym_n}p^*F)\cong \mu_*q_*^{\sym_n}(q^*E\otimes p^*F) &\cong \pi_*^{\sym_n}p_*(q^*E\otimes p^*F)\\&\cong \pi_*^{\sym_n}(p_*q^*E\otimes F)\\&\cong \pi^{\sym_n}_*(\Phi(E)\otimes F)\,.\qedhere 
\end{align*}
\end{proof}
%

%
%
\begin{cor}\label{cor:relativeFMcoh} $E\in \D(X^{[n]})$ and $F\in\D_{\sym_n}(X^n)$ there are functorial isomorphisms  
 \[
  \Ho^*(X^{[n]}, E\otimes \Psi(F))\cong\Ho(X^{(n)},\pi_*^{\sym_n}(\Phi(E)\otimes F))\cong \Ho_{\sym_n}^*(X^{n},\Phi(E)\otimes F)\,.
 \]
where $\Ho_{\sym_n}^*(X^{n},\Phi(E)\otimes F)$ denotes the equivariant cohomology, i.e.\
\[
\Ho_{\sym_n}^*(X^{n},\Phi(E)\otimes F)\cong \Hom_{\sym_n}^*(\reg_{X^{n}},\Phi(E)\otimes F)\cong \Ho^*\bigl(X^{n},\Res(\Phi(E)\otimes F)\bigr)^{\sym_n} \,.
\]
\end{cor}
\subsection{Euler characteristic of tensor products of tautological bundles}
In addition to the description of $\Phi(F^{[n]})$ of \cite{Sca1, Sca2} for $F\in \D(X)$ we now have a description of $\Psi^{-1}(\wedge^k L^{[n]})$ for $L\in \Pic X$ by \autoref{thm:wedgeline}.
Hence, as an application of the observation made in the previous subsection, we can compute the Euler characteristic of objects of the form $F^{[n]}\otimes \wedge^kL^{[n]}$. 
\begin{prop}\label{prop:mutensor}
For $F\in \D(X)$, $L\in \Pic X$, and $0\le k\le n$, we have natural isomorphisms
\[\mu_*(F^{[n]}\otimes \wedge^kL^{[n]})\cong \pi_*^{\sym_n}(\CC^\bullet_F\otimes \WW^k(L))\] with $\pi_*^{\sym_n}(\CC^p_F\otimes \WW^k(L))=0$ for $p>k$. Furthermore, 
$\pi_*^{\sym_n}(\CC^p_F\otimes \WW^k(L))$ is given by 
\begin{align*}
  p=0:&\quad\begin{cases}&\pi_*^{\sym_{k-1}\times \sym_{n-k}}\bigl(\pr_1^*F\otimes\pr_{[k]}^*(L^{\boxtimes k}\otimes \alt_k)\bigr)\\ &\oplus \pi_*^{\sym_{k}\times \sym_{n-k-1}}\bigl(\pr_{k+1}^*F\otimes\pr_{[k]}^*(L^{\boxtimes k}\otimes \alt_k)\bigr)\end{cases} \\
p\in [k-1]:&\quad\begin{cases} &\pi_*^{\sym_{p+1}\times\sym_{k-p-1}\times \sym_{n-k}}\bigl(F_{[p+1]}\otimes\pr_{[k]}^*(L^{\boxtimes k}\otimes \alt_k)\bigr)\\& \oplus \pi_*^{\sym_{p}\times \sym_{k-p}\times \sym_{n-k-1}}\bigl(F_{[p]\cup\{k+1\}}\otimes\pr_{[k]}^*(L^{\boxtimes k}\otimes \alt_k)\bigr)\end{cases}\\
p=k:&\quad\quad \pi_*^{\sym_{k}\times\sym_{n-k-1}}\bigl(F_{[k+1]}\otimes\pr_{[k]}^*(L^{\boxtimes k}\otimes \alt_k)\bigr) \,.\\
\end{align*}
\end{prop}
\begin{proof}
The assertion $\mu_*(F^{[n]}\otimes \wedge^kL^{[n]})\cong \pi_*^{\sym_n}(\CC^\bullet_F\otimes \WW^k(L))$ follows from  
\autoref{prop:relativeFMtensor} together with \autoref{thm:Scala} and \autoref{thm:wedgeline}. 

The rest is again an application of \autoref{rem:Danila}: We have 
\[
\CC^p_F\otimes \WW^k(L)\cong \bigoplus_{(I,J)\in \cI} F_I\otimes \pr_J^*(L^{\boxtimes k}\otimes \alt_k)\quad,\quad \cI=\{(I,J)\mid I,J\subset [n],\, |I|=p+1,\,,|J|=k\}\,.
\]
The $\sym_n$-action on $\CC^p_F\otimes \WW^k(L)$ induces the action $\sigma\cdot(I,J)=(\sigma(I), \sigma(J))$ on the index set $\cI$. Hence, the stabiliser subgroups are given by 
\[
 \sym_{(I,J)}=\sym_{I\cap J}\times \sym_{I\setminus J}\times \sym_{J\setminus I}\times \sym_{[n]\setminus(I\cup J)}\,.
 \]
Every transposition in $\sym_{I\setminus J}$ acts by $-1$ on $F_I\otimes \pr_J^*(L^{\boxtimes k}\otimes \alt_k)$. 
Hence, the $\sym_{I\setminus J}$-invariants of $F_I\otimes \pr_J^*(L^{\boxtimes k}\otimes \alt_k)$ vanish for $|I\setminus J|\ge 2$.
\end{proof}

\begin{remark} \label{rem:mutensor}
For $k=1$, the statement of \autoref{prop:mutensor} remains valid if we replace the line bundle $L$ by an arbitrary object $E\in \D(X)$; compare \autoref{thm:PsiCC}. Hence, we recover \cite[Thm.\ 3.2.2]{Sca1} and \cite[Thm.\ 32]{Sca2}. 
\end{remark}

\begin{theorem}\label{thm:tensorEuler}Let $F\in \D(X)$, $L\in \Pic X$, and $0\le k\le n$. If $X$ is projective,
 \begin{align*}
  \chi(F^{[n]}\otimes \wedge^kL^{[n]})=\begin{aligned}&s^{n-k-1}\chi(\reg_X)\cdot\bigl(\sum_{p=0}^k(-1)^p\chi(F\otimes L^{\otimes p})\cdot \lambda^{k-p}\chi(L)\bigr)\\& - s^{n-k}\chi(\reg_X)\cdot\bigl(\sum_{p=1}^k(-1)^p\chi(F\otimes L^{\otimes p})\cdot \lambda^{k-p}\chi(L)\bigr)\,.
                                                                              \end{aligned}
 \end{align*}
\end{theorem}
\begin{proof}
By \autoref{prop:mutensor}, we have $\Ho^*(F^{[n]}\otimes \wedge^k L^{[n]})\cong \Ho^*( \pi_*^{\sym_n}(\CC^\bullet_F\otimes \WW^k(L)))$. The assertion follows now from a straightforward computation of the Euler characteristics of the terms of $\pi_*^{\sym_n}(\CC^\bullet_F\otimes \WW^k(L))$. For example, for $p\in [k-1]$, the K\"unneth formula gives
\begin{align*}&\Ho^*\Bigl(\pi_*^{\sym_{p+1}\times\sym_{k-p-1}\times \sym_{n-k}}\bigl(F_{[p+1]}\otimes\pr_{[k]}^*(L^{\boxtimes k}\otimes \alt_k)\bigr)\Bigr)\\\cong &\Ho^*(F\otimes L^{\otimes p+1})\otimes \wedge^{k-p-1}\Ho^*(L)\otimes S^{n-k}\Ho^*(\reg_X)\,,\end{align*} 
hence
\[
 \chi\Bigl(\pi_*^{\sym_{p+1}\times\sym_{k-p-1}\times \sym_{n-k}}\bigl(F_{[p+1]}\otimes\pr_{[k]}^*(L^{\boxtimes k}\otimes \alt_k)\bigr)\Bigr)=\chi(F\otimes L^{\otimes p+1})\cdot \lambda^{k-p-1}\chi(L)\cdot s^{n-k}\chi(\reg_X).\qedhere
\]
\end{proof}

\begin{remark}\label{rem:tensorEuler}
 Again, it is possible to express \autoref{thm:tensorEuler} as an equality of generating functions, thus organising the formulae for varying $n$ and $k$ in one equation:
 \[
\sum_{n=0}^\infty\chi(F^{[n]}\otimes \Lambda_u L^{[n]})Q^n=\frac{(1+uQ)^{\chi(L)}}{(1-Q)^{\chi(\reg_X)}}\cdot \sum_{p=1}^\infty (-1)^{p-1}\chi\bigl(F\otimes (L^{p-1}u^{p-1}+L^pu^p)\bigr)Q^p\,. 
 \]
\end{remark}
\begin{remark}
We consider the special case that $F=L$. There is the Schur decomposition 
\[
 L^{[n]}\otimes \wedge^k L^{[n]}\cong \wedge^{k+1} L^{[n]}\oplus \mathbb S_{(2,1,\dots,1)} L^{[n]}\,;
\]
see e.g.\ \cite[eq.\ (6.9)]{FH}. We obtain a formula for the Euler characteristic $\chi(\mathbb S_{(2,1,\dots,1)} L^{[n]})$ of the Schur construction on the Hilbert scheme in terms of Euler characteristics of line bundles on the surface. Indeed,
\[
\chi(\mathbb S_{(2,1,\dots,1)} L^{[n]})=\chi(L^{[n]}\otimes \wedge^kL^{[n]})-\chi(\wedge^{k+1} L^{[n]}) 
\]
and we have such formulae for both terms on the right-hand side; see \autoref{cor:bichar} and \autoref{thm:tensorEuler}.
We can also obtain something slightly stronger, a description of $\mu_*(\mathbb S_{(2,1,\dots,1)} L^{[n]})$. Namely, one can check that, under the isomorphism of \autoref{prop:mutensor}, the direct summand $\mu_*(\wedge^{k+1} L^{[n]})$ of $\mu_*(L^{[n]}\otimes \wedge^kL^{[n]})$ corresponds to $\pi_*^{\sym_{k+1}\times \sym_{n-k-1}}(\pr_{[k+1]}^*(L^{\boxtimes k+1}\otimes \alt_{k+1}))$ embedded as a direct summand of $\pi_*^{\sym_{k}\times \sym_{n-k-1}}\bigl(\pr_{k+1}^*L\otimes\pr_{[k]}^*(L^{\boxtimes k}\otimes \alt_k)\bigr)\subset \pi_*^{\sym_n}(\CC^0_L\otimes \WW^k(L))$.
\end{remark}

\section{Further remarks}\label{sec:further}
Setting $K=L$ in \eqref{KLcounting} recovers the formula
\begin{align}\label{Lcounting}
\sum_{n=0}^\infty \chi(\Lambda_{-v}L^{[n]}, \Lambda_{-u} L^{[n]})Q^n= \exp\left(\sum_{r=1}^\infty \chi(\Lambda_{-v^r} L, \Lambda_{-u^r} L) \frac{Q^r}r    \right) 
\end{align}
which was shown in \cite{WangZhouTaut}. 
If we replace the surface $X$ by a quasi-projective variety $Y$ of arbitrary dimension, one can still associate to every vector bundle $E$ on $Y$ a tautological bundle $E^{[n]}$ on the Hilbert scheme $Y^{[n]}$ of $n$ points on $Y$ by means of the Fourier--Mukai transform along the universal family.
In \textit{loc.\ cit.\ }formula \eqref{Lcounting} is conjectured to generalise to smooth projective varieties of arbitrary dimension (instead of the surface $X$) and vector bundles of arbitrary rank (instead of the line bundle $L$). In the following, we give some restrictions to this conjecture. Namely, we prove that it does not hold if we replace $X$ by a smooth curve neither if we replace $L$ by a vector bundle of higher rank. For tautological bundles $L^{[n]}$ associated to line bundles on a smooth variety $Y$ with $\dim Y>2$, however, the conjecture still seems reasonable; see \autoref{rem:conj}. In this case, one can also hope that the slightly more general formula \eqref{KLcounting} still holds.

\subsection{Tautological bundles on Hilbert schemes of points on curves}
For $C$ a smooth curve, the Hilbert--Chow morphism $\mu\colon C^{[n]}\to C^{(n)}$ is an isomorphism. Under the identification $C^{[n]}\cong C^{(n)}$, the role of the universal family of $n$-points is played by $\Xi\cong C\times C^{(n-1)}\cong C^n/\sym_{n-1}$.
We define the reduced subscheme $D\subset C\times C^n$ as the polygraph
\[
 D=\bigl\{(x,x_1,\dots, x_n)\mid x=x_i \text{ for some $i=1,\dots,n$}\bigr\}\,.
\]
It is invariant under the $\sym_n$-action on $C\times C^n$ given by the permutation action on $C^n$.
\begin{lemma}
There is an isomorphism $D/\sym_n\cong \Xi$. 
\end{lemma}
\begin{proof}
We have $\Xi\cong C^n/\sym_{n-1}$. The irreducible components of $D$ are given by
\[
 C^n\cong D_i=\bigl\{(x,x_1,\dots,x_n)\mid x=x_i \bigr\}\subset C\times C^n
\]
for $i=1,\dots,n$. 
For $I\subset [n]$, we set $D_I:=\cap_{i\in I} D_i$.
Since the components intersect transversely, exactly as in the surface case, we get an $\sym_n$-equivariant resolution 
\begin{align}\label{ODresolution}
0\to \reg_D\to \bigoplus_{i=1}^n \reg_{D_i} \to \bigoplus_{|I|=2}\reg_{D_I}\to \dots \to \reg_{D_{[n]}}\to 0 
\end{align}
of $\reg_D$; see \cite[Rem.\ 2.2.1]{Sca1} or \autoref{rem:Postnikov} for details. The linearisations of the terms of this complex are given in such a way that, for $|I|\ge 2$, every transposition $(i\,\, j)$ with $i,j\in I$ acts by $-1$ on $\reg_{D_I}$. Hence, the $\sym_n$-invariants of the higher degree terms of \eqref{ODresolution} vanish and, if we denote the quotient morphism by $g\colon D\to D/\sym_n$, we have $g^{\sym_n}_*\reg_D\cong g^{\sym_n}_*(\oplus_{i=1}^n \reg_{D_i})$. By \autoref{rem:Danila}, we get 
\begin{align}\label{invaformula}
 g_*^{\sym_n}\reg_D\cong g^{\sym_{n-1}}_*\reg_{D_1}\cong g^{\sym_{n-1}}_*\reg_{C^n}\,.
\end{align}
There is a natural bijection between the $\sym_n$-orbits of $D$ and the $\sym_{n-1}$-orbits of $C^n$. Together with \eqref{invaformula}, this shows that $D/\sym_n\cong C^n/\sym_{n-1}\cong \Xi$. 
\end{proof}

\begin{lemma}
Let $F=\Xi \times_{C^{[n]}} C^n$ be the fibre product defined by the projection $\Xi\to C^{[n]}$ and the $\sym_n$-quotient morphism $C^n\to C^{[n]}$. Then there is an isomorphism $F\cong D$.
\end{lemma}
\begin{proof}
The fibre product $F$ is flat over the smooth variety $\Xi=C\times C^{(n-1)}$, hence Cohen--Macauley. It follows that $F$ is reduced since it is generically reduced. As a subset, the fibre product $F\subset C\times C^{(n-1)}\times C^n$ is given by
\[
 F=\bigl\{(x, x_2+\dots+x_n, y_1,\dots, y_n)\mid x+x_2+\dots+x_n=y_1+\dots+ y_n  \bigr\}\,.
\]
Hence, the projection $C\times C^{(n-1)}\times C^n\to C\times C^n$ induces a morphism $F\to D$ which is a bijection. We get the inverse morphism by applying the universal property of the fibre product to the projection $D\to C^n$ and the $\sym_n$-quotient morphism $D\to \Xi$.
\end{proof}
In complete analogy to the surface case, we define tautological objects using the Fourier--Mukai transform along the universal family as $F^{[n]}:=\FM_{\reg_\Xi}(F)\in \D(C^{[n]})$ for $F\in \D(C)$. We also define the $\sym_n$-equivariant objects $\CC^\bullet_F, \CC(F), \WW^k(F)\in \D_{\sym_n}(C^n)$ in the same way as in the surface case; see \autoref{subsection:Scala} and \autoref{def:CW}. 
Let $\pi\colon C^n\to C^{(n)}$ denote the quotient morphism.
Since the Hilbert--Chow morphism is an isomorphism in the curve case, we can interpret the functor $\pi^{\sym_n}_*\colon \D_{\sym_n}(C^n)\to \D(C^{(n)})$ as playing the role of $\Psi\colon \D_{\sym_n}(X^n)\to \D(X^{[n]})$ from the surface case. Also, $\pi^*\colon\D(C^{(n)})\to \D_{\sym_n}(C^n)$ plays the role of $\Phi\colon \D(X^{[n]})\to \D_{\sym_n}(X^n)$. However, these two functors are not equivalences in the curve case, but $\pi^*\colon \D(C^{(n)})\to \D_{\sym_n}(C^n)$ is still fully faithful; see \autoref{lem:quotientff}. 
\begin{prop}\label{prop:curvecase}For $E,F\in \D(C)$ and $L\in \Pic C$, we have 
\begin{align}
\label{item1} \pi^*F^{[n]}&\cong \CC_F^\bullet\,,\\
\label{item2} \pi^{\sym_n}_*\CC(F)&\cong F^{[n]}\,,\\  
\label{item3} \pi_*^{\sym_n}\WW^k(L)&\cong \wedge^k L^{[n]}\,,\\
\label{item4} E^{[n]}\otimes F^{[n]}&\cong \pi_*^{\sym_n}(\CC(E)\otimes \CC^\bullet_F)\,,\\
\label{item5} \sHom(E^{[n]}, F^{[n]})&\cong \pi_*^{\sym_n}\sHom(\CC^\bullet_E, \CC(F))\cong \pi_*^{\sym_n}\sHom(\CC(E), \CC^\bullet_F\otimes \alt_n)\,, \\ 
\label{item6} F^{[n]}\otimes \wedge^kL^{[n]}&\cong \pi_*^{\sym_n}(\CC^\bullet_F\otimes \WW^k(L))\,,\\
\label{item7} \sHom(E^{[n]}, \wedge^kL^{[n]})&\cong \pi_*^{\sym_n}\sHom(\CC^\bullet_E, \WW^k(L))\,,\\
\label{item8} \sHom(\wedge^kL^{[n]}, F^{[n]})&\cong \pi_*^{\sym_n}\sHom(\WW^k(L), \CC^\bullet_F\otimes \alt_n)\,.
\end{align}
\end{prop}
\begin{proof}
By the previous lemma, we have a cartesian diagram 
\[
\begin{xy}
\xymatrix{
 D \ar[r]  \ar[d] & C^n\ar^{\pi}[d]   \\
 \Xi \ar[r]\ar[d] & C^{[n]}=C^{(n)} \\
 C  & \,.
}
\end{xy} 
\]
Hence, by flat base change, we get $\pi^*F^{[n]}\cong \FM_{\reg_D}(F)$. Now, the proof of \eqref{item1} can be done using the resolution \eqref{ODresolution} in the same way as in the surface case; see \cite[Thm.\ 2.2.3]{Sca1} or \cite[Thm.\ 16]{Sca2} or \autoref{rem:Postnikov}.

For the verification of \eqref{item2}, we use the commutative diagram

\[
\begin{xy}
\xymatrix{
C^n \ar^{\pr_1}[rd]  \ar^{q_1}[d] \ar@/_18mm/_\pi[dd] &    \\
 \Xi=C\times C^{(n-1)} \ar_{\quad\quad\pr_C}[r]\ar^{\pr_{C^{[n]}}}[d] & C \\
 C^{[n]}  & 
}
\end{xy} 
\]
where $q_1$ is the $\sym_{n-1}$-quotient morphism and imitate the proof of \autoref{prop:Psi}.

Similarly, the proof of \eqref{item3} can be done in analogy to the proof of \autoref{thm:wedgeline} as given in \autoref{sec:tautBKRH}.

Formula \eqref{item4} follows from \eqref{item2}, the equivariant projection formula, and \eqref{item1}:
\[
E^{[n]}\otimes F^{[n]}\cong \pi_*^{\sym_n}\CC(E)\otimes F^{[n]}\cong \pi_*^{\sym_n}(\CC(E)\otimes \pi^*F^{[n]})\cong \pi_*^{\sym_n}(\CC(E)\otimes \CC_F^\bullet)\,.
\]
The verification of the first isomorphism of \eqref{item5} is basically the same. For the second isomorphism, note that the equivariant relative canonical bundle of the quotient is given by $\omega_\pi\cong \reg_{C^n}\otimes \alt$; see \cite[Lem.\ 5.10]{KSos}. Hence, by equivariant Grothendieck duality, we get
\begin{align*}
\sHom(E^{[n]}, F^{[n]})\cong\sHom(\pi_*^{\sym_n}\CC(E), F^{[n]})&\cong\pi_*^{\sym_n}\sHom(\CC(E), \pi^!F^{[n]})\\&\cong \pi_*^{\sym_n}\sHom(\CC(E), \CC^\bullet_F\otimes \alt_n)\,. 
\end{align*}
The verifications of \eqref{item6}, \eqref{item7}, and \eqref{item8} are analogous to those of \eqref{item4} and \eqref{item5} using \eqref{item3} instead of \eqref{item2}.
\end{proof}
Now, we can apply the global section functor to both sides of the formulae of 
\autoref{prop:curvecase} to obtain formulae for the homological invariants of tautological sheaves on $C^{[n]}$ in terms of homological invariants on the curve $C$. The formulae for the cohomologies and their Euler characteristics are exactly the same as in the surface case. The reason is that \eqref{item2} parallels \autoref{thm:PsiCC}, \eqref{item3} parallels \autoref{thm:wedgeline}, \eqref{item4} parallels \autoref{rem:mutensor}, and \eqref{item6} parallels \autoref{prop:mutensor}.
The formulae for the extension groups and their Euler (bi-)characteristics, however, differ from the surface case. 

\begin{prop}\label{prop:Eulercurve}
 For $E,F\in \D(C)$, we have 
 \[
  \chi(E^{[n]}, F^{[n]})=\chi(E,F)\bigl(\sum_{p=0}^{n-1}(-1)^p\lambda^{n-1-p}\chi(\reg_C)\bigr) + \chi(E^\vee)\chi(F)\bigl(\sum_{p=0}^{n-2}(-1)^p\lambda^{n-2-p}\chi(\reg_C)\bigr)\,.
 \]
\end{prop}
\begin{proof}
For $I\subset [n]$, we denote by $F_I^+$ the $\sym_I\times \sym_{[n]\setminus I}$-equivariant object $\iota_{I*}p_I^*F$ with $\sym_I$ acting trivially (recall that $\sym_I$ acts on $F_I$ by $\alt_I$; see \autoref{subsection:Scala}). Using \autoref{rem:Danila}, we compute the degree $p$ terms of $\pi_*^{\sym_n}(\CC(E), \CC_F^\bullet\otimes \alt)\cong \sHom(E,F)$ as 
\begin{align*}
p=0:&\quad\quad\pi_*^{\sym_{n-1}}\bigl(\pr_1^*\sHom(E,F)\otimes  \alt_{[2,n]})\bigr)\oplus \pi_*^{\sym_{n-2}}\bigl(\pr_1^*E^\vee\otimes\pr_{2}^*F\otimes\alt_{[3,n]})\bigr)\,, \\
p\in[n-2]:&\quad\begin{cases}&\pi_*^{\sym_{p}\times \sym_{n-p-1}}\bigl(\pr_1^*E^\vee\otimes F_{[p+1]}^+\otimes   \alt_{[p+2,n]})\bigr)\\&\oplus \pi_*^{\sym_{p+1}\times \sym_{n-p-2}}\bigl(\pr_{p+2}^*E^\vee\otimes F_{[p+1]}^+\otimes   \alt_{[p+3,n]})\bigr)\,,\end{cases}\\
p=n-1:&\quad\quad\pi_*^{\sym_{n-1}}\bigl(\pr_1^*E^\vee\otimes F_{[n]}^+\bigr)\,.
\end{align*}
We can compute the cohomology of these terms in order to get the asserted formula for the Euler characteristic. For example, we have
\[
\Ho^*\bigl(\pi_*^{\sym_{p}\times \sym_{n-p-1}}(\pr_1^*E^\vee\otimes F_{[p+1]}^+\otimes   \alt_{[p+2,n]})\bigr) \cong \Hom^*(E,F)\otimes \wedge^{n-p-1}\Ho^*(\reg_C)\,.\qedhere
\]
\end{proof}
\begin{remark}\label{rem:conj}
The formula of \autoref{prop:Eulercurve} differs from the one of the surface case which, by \eqref{ExtEF}, reads
\[
\chi(E^{[n]}, F^{[n]})=\chi(E,F)s^{n-1}\chi(\reg_X) + \chi(E^\vee)\chi(F)s^{n-1}\chi(\reg_X)\,. 
\]
As mentioned in the introduction, taking $E=F=L\in \Pic C$, this implies that \cite[Conj.\ 1]{WangZhouTaut}, which is known to be true in the surface case, cannot hold for curves. However, in \cite[Sect.\ 6]{WangZhouTaut}, there is some evidence given for the conjecture to hold for tautological bundles on the Hilbert scheme $Y^{[n]}$ for $Y$ smooth of dimension $\dim Y >2$. We can add a further small piece of evidence to this as follows. We consider the case $n=2$. Then the Hilbert square $Y^{[2]}$ is smooth for $Y$ of arbitrary dimension. Furthermore, for $\dim Y>2$, the functor $\Psi\colon \D_{\sym_2}(Y^2)\to\D(Y^{[2]})$ is still fully faithful (but not an equivalence any more); see \cite{KPScyclic}. 
Note that, for $n=2$, we have $\cZ\cong \Xi\cong \Xi\binom 22$.
Using this, one can check that \autoref{thm:PsiCC} and \autoref{thm:wedgeline} remain valid for $\dim Y>2$ and $n=2$. Concretely, this means that $\Psi(\reg_{Y^2})\cong \reg_{Y^{[2]}}$, $\Psi(\CC(F))\cong F^{[2]}$ for $F\in \D(X)$, and $\Psi(L^{\boxtimes 2}\otimes \alt)\cong \det L^{[2]}$ for $L\in \Pic Y$. Hence, by the fully faithfulness of $\Psi$, the formulae of \autoref{cor:Extformulae} remain valid for $n=2$ and $\dim Y>2$.
\end{remark}


\subsection{Wedge powers of tautological bundles of higher rank}
In \cite[Sect.\ 2.3]{WangZhouTaut}, it is conjectured that formula \eqref{Lcounting} generalises from line bundles to vector bundles of arbitrary rank. The following example shows that this cannot hold, even in the surface case. Indeed, if $\rank F$ is odd, formula \eqref{Lcounting} with $L$ replaced by $F$ predicts that $\chi(\det F^{[2]})=\lambda^2\chi(\det F)$. However, we have the following
\begin{prop}
Let $X$ be a smooth projective surface and $F=\reg_X^{\oplus 3}$. Then 
\[
 \chi_{X^{[2]}}(\det F^{[2]})=\lambda^2\chi(\reg_X)-\chi(\Omega_X)\,.
\]
\end{prop}
\begin{proof}
By \cite{Scadiagonal}, we have $\Phi(\det F^{[2]})\cong \cI_{\Delta}^3\otimes \alt$. Using the short exact sequences 
\[
 0\to \cI^{i+1}_\Delta\to \cI^{i}_\Delta \to \cI^{i}_\Delta/\cI^{i+1}_\Delta\to 0\,, 
\]
we get 
\[
 \chi(\det F^{[2]})=\chi(\cI_\Delta^3\otimes \alt)=\chi(\reg_{X^2}\otimes \alt)-\chi(\reg_{X^2}/\cI_\Delta\otimes \alt)- \chi(\cI_\Delta/\cI^2_\Delta\otimes \alt) - \chi(\cI^2_\Delta/\cI^3_\Delta\otimes \alt)\,.
\]
where the terms on the right-hand side are the Euler characteristics of the equivariant cohomology. 
Since the natural action of $\sym_2$ on $\cI_\Delta^i/\cI_\Delta^{i+1}$ is given by $\alt^i$, the invariants $\pi_*^{\sym_2}(\reg_{X^2}/\cI_\Delta\otimes \alt)$ and $\pi_*^{\sym_2}(\cI^2_\Delta/\cI^3_\Delta\otimes \alt)$ vanish. Accordingly, also the terms 
$\chi(\reg_{X^2}/\cI_\Delta\otimes \alt)$ and $\chi(\cI^2_\Delta/\cI^3_\Delta\otimes \alt)$ vanish and we get the assertion.
\end{proof}

\appendix{\section{Computations with power series}\label{Appendix}

Given a power series $F(Q)$, we denote by $F(Q)_{\mid Q^n}$ the coefficient of $Q^n$.
With this notation, the verification that the two formulae of \autoref{cor:bichar} are equivalent comes down to the following 
\begin{prop}\label{prop:App}
\begin{align*}
&(-1)^{k+\ell}\exp\left(\sum_{r=1}^\infty \chi(\Lambda_{-v^r} K, \Lambda_{-u^r} L) \frac{Q^r}r    \right)_{\mid v^ku^\ell Q^n} \\
=& \sum_{i=\max \{0,k+\ell-n\}}^{\min\{k,\ell\}}s^i\chi(K,L)\cdot \lambda^{k-i}\chi(K^\vee)\cdot \lambda^{\ell-i}\chi(L)\cdot s^{n+i-k-\ell}\chi(\reg_X)
\end{align*}
\end{prop}

For the proof, we use two simple auxiliary lemmas.

\begin{lemma}\label{lem:App1}
 \[\exp\bigl(\sum_{r=1}^\infty \frac 1r Q^r\bigr) =\frac 1{1-Q}\,.\]
\end{lemma}
\begin{proof}
 One way to see this is to apply the logarithm to both sides.
\end{proof}
\begin{lemma}\label{lem:App2} For $k\in \IN$ and $\chi\in \IC$, we have 
\begin{enumerate}
 \item $s^k\chi=(-1)^k \lambda^k(-\chi)$,
 \item $(1+Q)^\chi_{\mid Q^k}=\lambda^k \chi$,
 \item $(1-Q)^{-\chi}_{\mid Q^k}=s^k\chi$.
\end{enumerate}
\end{lemma}
\begin{proof}
The verification of (i) is a direct computation using \autoref{sdefin} of the numbers $s^k\chi$ and $\lambda^k\chi$. Part (ii) is the binomial coefficient theorem. Part (iii) follows from (i) and (ii).
\end{proof}

\begin{proof} [Proof of \autoref{prop:App}] 
We have 
\[
 \chi(\Lambda_{-v^r} K, \Lambda_{-u^r} L) Q^r=\chi(K,L)(vuQ)^r -\chi(K^\vee)(vQ)^r -\chi(L)(uQ)^r+ \chi(\reg_X)Q^r\,.
\]
Hence, by \autoref{lem:App1}, we get 
\[
\exp\left(\sum_{r=1}^\infty \chi(\Lambda_{-v^r} K, \Lambda_{-u^r} L) \frac{Q^r}r    \right)
=(1-vuQ)^{-\chi(K,L)}  (1-vQ)^{\chi(K^\vee)}(1-uQ)^{\chi(L)}(1-Q)^{-\chi(\reg_X)}.
\]
Now, the assertion follows using \autoref{lem:App2}.
\end{proof}
The verification of \autoref{rem:tensorEuler} is very similar.

}

\bibliographystyle{alpha}
\addcontentsline{toc}{chapter}{References}
\bibliography{references}

\end{document}